\documentclass[11pt]{amsart}
\usepackage{amsmath,amssymb,amsthm,mathtools}
\usepackage[margin=1.15in]{geometry}
\usepackage[colorlinks=true,linkcolor=blue,citecolor=blue,urlcolor=blue]{hyperref}

\theoremstyle{plain}
\newtheorem{theorem}{Theorem}
\newtheorem{lemma}[theorem]{Lemma}

\newtheorem{corollary}[theorem]{Corollary}
\theoremstyle{remark}
\newtheorem{remark}[theorem]{Remark}

\numberwithin{equation}{section}

\DeclareMathOperator{\wt}{wt}
\DeclareMathOperator{\CT}{CT}
\DeclareMathOperator{\Newt}{Newt}
\DeclareMathOperator{\supp}{supp}
\newcommand{\C}{\mathbb{C}}
\newcommand{\Q}{\mathbb{Q}}
\newcommand{\Z}{\mathbb{Z}}
\newcommand{\N}{\mathbb{N}}
\newcommand{\E}{\mathbb{E}}
\newcommand{\Qbar}{\overline{\Q}}
\newcommand{\GMC}{\mathrm{GMC}}

\begin{document}

\title[A face-isolation proof of GMC(2)]
      {A face-isolation proof of the two-variable\\ Gaussian Moments Conjecture}

\author{Michael Wilson \\ \small University of Arkansas at Little Rock}

\date{\today}



\begin{abstract}
Let $X,Y$ be independent standard real Gaussian random variables and let
$P\in\C[X,Y]$ satisfy $\E(P^m)=0$ for every $m\ge1$. We prove that, in the
complex coordinates $Z=(X+iY)/\sqrt2$ and $W=(X-iY)/\sqrt2$, every monomial of
$P$ has strictly positive weight $\deg_Z-\deg_W$, or every monomial has
strictly negative weight. Consequently $\E(QP^m)=0$ for every $Q\in\C[X,Y]$ and
every $m>\deg Q$, which is the two-variable Gaussian Moments Conjecture with an
explicit threshold. Combined with the known case $n=1$ and Long's
counterexamples for $n\ge3$, this completes the classification of the
dimensions in which the conjecture holds. The main new ingredient is a prime
isolation theorem: for a suitable lower face of the Newton polygon, a
$p$-adic valuation count at moment index $m=qp$ separates the pure-face
contribution from all others and forces the constant terms of the powers of the
face weight polynomial to vanish, after which the theorem of Duistermaat and
van der Kallen applies.
\end{abstract}

\maketitle

\section{Introduction}\label{sec:intro}

Let $X=(X_1,\dots,X_n)$ be a vector of independent standard real Gaussian random
variables. Derksen, van den Essen and Zhao \cite{DvdEZ} formulated the Gaussian
Moments Conjecture $\GMC(n)$: if $P\in\C[x_1,\dots,x_n]$ satisfies
\begin{equation}\label{eq:hyp}
  \E\bigl(P(X)^m\bigr)=0 \qquad (m\ge1),
\end{equation}
then for every $Q\in\C[x_1,\dots,x_n]$ one has $\E\bigl(Q(X)P(X)^m\bigr)=0$ for
all sufficiently large $m$. Equivalently, the space of polynomials with
vanishing Gaussian expectation is a Mathieu--Zhao space in the sense of Mathieu
\cite{Mathieu}. Complex coefficients are essential: for real $P$ the case $m=2$
already forces $P=0$. Derksen, van den Essen and Zhao proved that $\GMC(n)$ for
all $n$ implies the Jacobian Conjecture \cite[Thm.~1.6]{DvdEZ}.

The status of the conjecture changed in July 2026. It holds for $n=1$
\cite[Prop.~4.2]{DvdEZ}. In the opposite direction, Long \cite{Long} exhibited
explicit polynomials showing that $\GMC(n)$ fails for every $n\ge3$: a six-term
cubic in four variables settles $n\ge4$, and a five-term quartic in three
variables settles $n\ge3$. The only dimension left unresolved was $n=2$.

Throughout we work in the complex coordinates
\begin{equation}\label{eq:coords}
  Z=\frac{X+iY}{\sqrt2},\qquad W=\frac{X-iY}{\sqrt2},
\end{equation}
which form an invertible linear change of coordinates, so that
$\C[X,Y]=\C[Z,W]$ and total degree is the same in either. As random variables
$W=\overline{Z}$; both are complex linear polynomials in $X,Y$. We assign the
weight
\[
  \wt(Z^aW^b)=a-b,
\]
so that Gaussian expectation annihilates every monomial of nonzero weight; see
\eqref{eq:moments} below.

Two classes of $P$ were already known to satisfy the conclusion. If $P$ is
homogeneous, $\GMC(2)$ holds for $P$ \cite[Cor.~4.4]{DvdEZ}, by an argument
passing through the theorem of Duistermaat and van der Kallen on Laurent
polynomials \cite{DvdK}. If $P=ZA(U)+WB(U)$ with $U=ZW$, then Long
\cite[\S7]{Long} shows that \eqref{eq:hyp} forces $A=0$ or $B=0$. In both cases
the surviving polynomials are supported on monomials of a single sign of weight,
and Long concludes that a counterexample to $\GMC(2)$ would have to involve a
richer inhomogeneous mixture of positive and negative weights. The purpose of
this paper is to show that no such mixture exists.

\begin{theorem}\label{thm:main}
Let $X,Y$ be independent standard real Gaussian random variables and let
$P\in\C[X,Y]$ satisfy $\E(P^m)=0$ for every $m\ge1$. Write $P$ in the
coordinates \eqref{eq:coords}. Then either every monomial occurring in $P$ has
weight at least $1$, or every monomial occurring in $P$ has weight at most
$-1$. Consequently, for every $Q\in\C[X,Y]$,
\[
  \E\bigl(QP^m\bigr)=0 \qquad\text{whenever } m>\deg Q .
\]
In particular $\GMC(2)$ holds, with the explicit threshold $m\ge\deg Q+1$.
\end{theorem}

Together with $\GMC(1)$ \cite[Prop.~4.2]{DvdEZ} and \cite{Long}, this settles
the conjecture in every dimension: $\GMC(n)$ is true for $n\le2$ and false for
$n\ge3$.

The proof occupies Sections \ref{sec:prelim}--\ref{sec:proof}. Section
\ref{sec:prelim} records the moment evaluation and the elementary weight bound
that supplies the threshold. Section \ref{sec:spec} reduces to algebraic
coefficients; the point of that reduction, emphasised in Remark
\ref{rem:supportonly}, is that everything proved afterwards about the Newton
polygon is a statement about the \emph{support} of $P$ alone, so it transfers
back to the original complex polynomial. Section \ref{sec:isolation} contains
the main new ingredient, a prime isolation theorem for an arbitrary lower face
of the Newton polygon. Section \ref{sec:edge} shows that a Newton polygon
carrying weights of both signs always exposes a face to which prime isolation
applies, and Section \ref{sec:proof} assembles the proof. Section
\ref{sec:remarks} explains why the argument is specific to $n=2$ and records two
corollaries.

\section{Coordinates, moments and weights}\label{sec:prelim}

Throughout, $X$ and $Y$ are independent standard real Gaussians and $Z,W$ are as
in \eqref{eq:coords}. We use repeatedly the evaluation
\begin{equation}\label{eq:moments}
  \E\bigl(Z^aW^b\bigr)=\delta_{ab}\,a! \qquad (a,b\ge0).
\end{equation}
Indeed, the law of $(X,Y)$ is invariant under rotation, which multiplies
$Z^aW^b$ by $e^{i(a-b)\theta}$; this forces the expectation to vanish when
$a\ne b$. When $a=b$, the variable $ZW=(X^2+Y^2)/2$ is exponentially
distributed with mean one, whence $\E\bigl((ZW)^a\bigr)=a!$.

\begin{remark}
We stress that $Z$ and $W$ are \emph{not} an independent pair: as random
variables $W=\overline Z$, and $\E(Z^aW^a)=a!$ rather than
$\bigl((a-1)!!\bigr)^2$. Equation \eqref{eq:moments} is an identity about two
independent \emph{real} Gaussians written in complex linear coordinates; no
complex Gaussian measure is used anywhere in this paper.
\end{remark}

We write $\supp P\subset\N^2$ for the set of exponent pairs occurring in $P$,
and $\Newt(P)=\mathrm{conv}(\supp P)\subset\mathbb R_{\ge0}^2$ for its Newton
polygon.

\begin{lemma}[One-sided weights]\label{lem:weights}
Let $P\in\C[X,Y]$ be such that every monomial occurring in $P$ has weight at
least $1$, or such that every monomial occurring in $P$ has weight at most
$-1$. Then $\E(P^m)=0$ for every $m\ge1$, and for every $Q\in\C[X,Y]$,
\[
  \E\bigl(QP^m\bigr)=0 \qquad\text{whenever } m>\deg Q .
\]
\end{lemma}

\begin{proof}
Assume every monomial of $P$ has weight at least $1$. Weight is additive on
products, so every monomial occurring in $P^m$ has weight at least $m$. Write
$Q=\sum_{i,j}q_{ij}Z^iW^j$; each monomial of $Q$ has weight $i-j$ with
$|i-j|\le i+j\le\deg Q$. Hence every monomial of $QP^m$ has weight at least
$m-\deg Q$, which is strictly positive as soon as $m>\deg Q$. By
\eqref{eq:moments} Gaussian expectation annihilates each such monomial, so
$\E(QP^m)=0$. Taking $Q=1$ gives $\E(P^m)=0$ for $m\ge1$. The case of weights at
most $-1$ follows by exchanging $Z$ and $W$, which is the substitution
$Y\mapsto-Y$ and leaves the law of $(X,Y)$ invariant. The zero polynomial
satisfies both hypotheses vacuously.
\end{proof}

\section{Reduction to algebraic coefficients}\label{sec:spec}

\begin{lemma}[Algebraic specialization]\label{lem:spec}
Fix a finite set $S\subset\N^2$. Suppose there exists
$P=\sum_{\alpha\in S}c_\alpha Z^{a_\alpha}W^{b_\alpha}\in\C[Z,W]$ with
$c_\alpha\ne0$ for every $\alpha\in S$ and $\E(P^m)=0$ for every $m\ge1$. Then
there exists such a polynomial with all coefficients in $\Qbar$.
\end{lemma}

\begin{proof}
By \eqref{eq:moments}, for each $m\ge1$,
\begin{equation}\label{eq:momentexp}
  \E(P^m)=\sum_{\substack{k\in\N^{S},\ |k|=m\\ A(k)=B(k)}}
  \binom{m}{k}\Bigl(\prod_{\alpha\in S}c_\alpha^{k_\alpha}\Bigr)A(k)!,
  \qquad
  A(k)=\sum_\alpha k_\alpha a_\alpha,\quad B(k)=\sum_\alpha k_\alpha b_\alpha,
\end{equation}
which is a polynomial $\varphi_m\in\Z[c_\alpha:\alpha\in S]$ in the
coefficients. Let $I\subset\Q[c_\alpha]$ be the ideal generated by
$\{\varphi_m\}_{m\ge1}$; it is finitely generated by the Hilbert basis theorem.
Let $f=\prod_{\alpha\in S}c_\alpha$. The set
$V(I)\cap D(f)$ is a constructible set defined over $\Q$, and by hypothesis it
has a $\C$-point. Since $\Qbar$ is algebraically closed and
$\Qbar\subset\C$, the Nullstellensatz gives a $\Qbar$-point of the same set,
which is the required polynomial.
\end{proof}

\begin{remark}[Why this reduction is legitimate]\label{rem:supportonly}
Lemma \ref{lem:spec} produces a \emph{different} polynomial, so it is worth
recording precisely how it will be used. Step 1 of Section \ref{sec:proof}
argues by contradiction: from a complex $P$ with support $S$ carrying weights of
both signs and satisfying \eqref{eq:hyp}, Lemma \ref{lem:spec} supplies
$P'\in\Qbar[Z,W]$ with the same support and the same vanishing, and Lemmas
\ref{lem:isolation} and \ref{lem:edge} applied to $P'$ then contradict each
other. Both of those conclusions are assertions about the exposed face
$F\subset S$ alone ,  which weights occur on $F$ ,  and mention no
coefficients, so the contradiction is not attached to the particular
$\Qbar$-point produced. What transfers the conclusion back to $P$ is simply the
contrapositive of Lemma \ref{lem:spec}: no polynomial over $\Qbar$ with support
$S$ satisfies \eqref{eq:hyp}, hence none over $\C$ does either. This is the only
use we make of Lemma \ref{lem:spec}.
\end{remark}

\section{Prime isolation of a lower face}\label{sec:isolation}

Let $P\in\Qbar[Z,W]$ with $\supp P=S$, let $K=\Q(c_\alpha:\alpha\in S)$, a
number field, and let $\mathcal O_K$ be its ring of integers. For an integral
linear form
\[
  \ell(a,b)=ca+db, \qquad c,d\in\Z,\quad c+d>0,
\]
put $\lambda=\min_{\alpha\in S}\ell(\alpha)$ and let
$F=\{\alpha\in S:\ell(\alpha)=\lambda\}$ be the exposed face. We may and do
assume $\gcd(c,d)=1$: replacing $(c,d)$ by $(c,d)/\gcd(c,d)$ scales $\lambda$
and $c+d$ by the same factor, leaves $F$ and $\lambda/(c+d)$ unchanged, and only
shrinks the set of primes excluded below. Define the \emph{face weight
polynomial}
\[
  R(z)=\sum_{\alpha\in F}c_\alpha z^{a_\alpha-b_\alpha}
  \ \in\ K[z,z^{-1}].
\]

\begin{remark}[No cancellation in $R$]\label{rem:injective}
The weight $w(a,b)=a-b$ is injective on $F$. Indeed, if $\alpha,\alpha'\in F$
satisfy $w(\alpha)=w(\alpha')$, then $v=\alpha-\alpha'$ lies in
$\ker\ell\cap\ker w$, so $cv_1+dv_2=0$ and $v_1=v_2$, whence $(c+d)v_1=0$ and
$v=0$ because $c+d>0$. Equivalently, $\ker\ell=\langle(d,-c)\rangle$ and
$\ker w=\langle(1,1)\rangle$ meet only in the origin, since
$\det\begin{psmallmatrix}d&-c\\1&1\end{psmallmatrix}=c+d\ne0$.

Consequently $R$ has exactly $|F|$ terms, with pairwise distinct exponents and
with the nonzero coefficients $c_\alpha$: no cancellation can occur, and the
map $\alpha\mapsto w(\alpha)$ is a bijection from $F$ onto the exponent set of
$R$. This is what licenses the passage, in Lemma \ref{lem:isolation}, from a
statement about the exponents of $R$ to a statement about the weights occurring
on $F$. Without it a weight-zero point of $F$ could in principle be cancelled
out of $R$ by another face point of the same weight, and the contradiction in
Step 1 of Section \ref{sec:proof} would fail in exactly the sub-case where $w$
vanishes somewhere on $F$.
\end{remark}

\begin{lemma}[Prime isolation]\label{lem:isolation}
With the notation above, suppose $\E(P^m)=0$ for every $m\ge1$ and
$\lambda/(c+d)\ge0$. Let $B\ge1$ be the denominator of $\lambda/(c+d)$ in lowest
terms. Then
\begin{equation}\label{eq:CTvanish}
  \CT R(z)^{Bn}=0 \qquad (n\ge1).
\end{equation}
Consequently every exponent occurring in $R$ is strictly positive, or every
exponent occurring in $R$ is strictly negative. By Remark \ref{rem:injective}
this says equivalently that all weights occurring on $F$ have the same strict
sign.
\end{lemma}

\begin{proof}
Fix $n\ge1$ and set $q=Bn$, so that
\[
  R_0:=\frac{q\lambda}{c+d}\in\N .
\]
Call a rational prime $p$ \emph{good} (for the data $P,\ell,q$) if
\begin{enumerate}
\item[(G1)] $p$ is unramified in $K$;
\item[(G2)] every $c_\alpha$ is $\mathfrak p$-integral for every prime
$\mathfrak p$ of $\mathcal O_K$ above $p$;
\item[(G3)] $p\nmid(c+d)$;
\item[(G4)] $p>R_0$.
\end{enumerate}
All but finitely many primes satisfy (G1)--(G3), and (G4) excludes finitely many
more once $q$ is fixed; so infinitely many good primes exist. Note that (G2) asks
only for integrality, not for the reductions $\overline{c_\alpha}$ to be
nonzero: the moment computation below takes place in characteristic zero and is
reduced only at the very end, so a coefficient meeting $\mathfrak p$ does no
harm. Fix a good $p$, fix $\mathfrak p\mid p$, and let $v=v_{\mathfrak p}$ be
normalised by $v(p)=1$, which is possible by (G1). Set $m=qp$.

Consider the expansion \eqref{eq:momentexp} of $\E(P^m)$. For a multi-index $k$
occurring there we have $A(k)=B(k)=:N$, and therefore
\begin{equation}\label{eq:balance}
  (c+d)N=cA(k)+dB(k)=\sum_{\alpha}k_\alpha\ell(\alpha).
\end{equation}
Since $\ell(\alpha)\ge\lambda$ for every $\alpha\in S$, \eqref{eq:balance} gives
$(c+d)N\ge m\lambda=qp\lambda$, hence
\begin{equation}\label{eq:Nlower}
  N\ \ge\ pR_0 .
\end{equation}

\smallskip
\noindent\emph{The pure-face contribution.} Suppose $k_\alpha=0$ for every
$\alpha\notin F$. Then \eqref{eq:balance} reads $(c+d)N=qp\lambda$, so
$N=pR_0$, independently of $k$. Moreover the balance condition $A(k)=B(k)$ is
exactly $\sum_\alpha k_\alpha(a_\alpha-b_\alpha)=0$, which is the condition
selecting the constant term of $R^{m}$. Hence the sum of all pure-face terms in
\eqref{eq:momentexp} equals
\begin{equation}\label{eq:pureface}
  (pR_0)!\,\CT R(z)^{qp}.
\end{equation}

\smallskip
\noindent\emph{Every other contribution is divisible by $p^{R_0+1}$.} Suppose
$k_\beta>0$ for some $\beta\notin F$. We distinguish two cases.

\emph{Case 1: $p\nmid k_\alpha$ for some $\alpha$.} The base-$p$ units digit of
$m=qp$ is $0$, while the units digits of the $k_\alpha$ sum to a multiple of
$p$ which is positive, hence at least $p$. Adding the $k_\alpha$ in base $p$
therefore produces a carry, and by Kummer's theorem for multinomial
coefficients, $v\bigl(\binom{m}{k}\bigr)\ge1$. By \eqref{eq:Nlower},
\[
  v(N!)\ \ge\ \Bigl\lfloor\frac{N}{p}\Bigr\rfloor\ \ge\ R_0 ,
\]
so the term has valuation at least $R_0+1$.

\emph{Case 2: $k_\alpha=p\,l_\alpha$ for every $\alpha$.} Then
$\sum_\alpha l_\alpha=q$, and
\[
  \Delta:=\sum_\alpha l_\alpha\bigl(\ell(\alpha)-\lambda\bigr)
\]
is a positive integer, since $\ell(\beta)>\lambda$ and $l_\beta>0$. Now
\eqref{eq:balance} gives $(c+d)N=p(q\lambda+\Delta)$. As $q\lambda=R_0(c+d)$ is
divisible by $c+d$, we get $(c+d)\mid p\Delta$, and $p\nmid(c+d)$ by (G3),
whence $(c+d)\mid\Delta$ and $\Delta/(c+d)\ge1$. Therefore
\[
  N=p\Bigl(R_0+\frac{\Delta}{c+d}\Bigr)\ \ge\ p(R_0+1),
  \qquad
  v(N!)\ \ge\ \Bigl\lfloor\frac{N}{p}\Bigr\rfloor\ \ge\ R_0+1 ,
\]
so again the term has valuation at least $R_0+1$.

\smallskip
\noindent\emph{Conclusion of the valuation count.} By hypothesis
$\E(P^{qp})=0$. Subtracting the contributions just bounded, \eqref{eq:pureface}
gives
\[
  (pR_0)!\,\CT R^{qp}\ \equiv\ 0 \pmod{\mathfrak p^{R_0+1}} .
\]
Since $p>R_0$ by (G4), we have
$v\bigl((pR_0)!\bigr)=\lfloor pR_0/p\rfloor=R_0$, and dividing by $p^{R_0}$
yields
\[
  \CT R^{qp}\ \equiv\ 0 \pmod{\mathfrak p} .
\]
Let $\kappa=\mathcal O_K/\mathfrak p$, a field of characteristic $p$, and let
$\overline R\in\kappa[z,z^{-1}]$ be the reduction of $R$, which is defined by
(G2). Writing $\overline R=\sum_i \bar a_iz^i$ we have
$\overline R^{\,p}=\sum_i\bar a_i^{\,p}z^{ip}$, whence
\[
  \CT \overline R^{\,qp}
  =\CT\bigl(\overline R^{\,q}\bigr)^{p}
  =\bigl(\CT \overline R^{\,q}\bigr)^{p} .
\]
(We emphasise that the correct identity is $\CT(g^p)=(\CT g)^p$; the variant
$\CT(g^p)=\CT(g)$ is valid only when the coefficients lie in the prime field,
which need not hold here.) Since $\kappa$ is a field, it follows that
$\CT R^{q}\equiv0\pmod{\mathfrak p}$. This holds for infinitely many good
primes $p$, and an element of $K$ lying in infinitely many distinct primes of
$\mathcal O_K$ is zero; hence $\CT R^{q}=0$. As $n\ge1$ was arbitrary and
$q=Bn$, this proves \eqref{eq:CTvanish}.

\smallskip
\noindent\emph{Application of Duistermaat--van der Kallen.} Put $g=R^{B}$, a
Laurent polynomial in one variable. By \eqref{eq:CTvanish},
$\CT(g^{n})=\CT R^{Bn}=0$ for every $n\ge1$. By the theorem of Duistermaat and
van der Kallen \cite{DvdK}, a Laurent polynomial in one variable which is
neither a polynomial in $z$ nor a polynomial in $z^{-1}$ has some positive power
with nonzero constant term; combined with $\CT(g)=0$ this forces
$g\in z\C[z]$ or $g\in z^{-1}\C[z^{-1}]$, equivalently $0\notin\Newt(g)$. Since
the extreme coefficients of $R^{B}$ are powers of the extreme coefficients of
$R$ and hence nonzero, $\Newt(R^{B})=B\cdot\Newt(R)$, and as $B\ge1$ we conclude
$0\notin\Newt(R)$. In one variable this says exactly that every exponent of $R$
is strictly positive or every exponent is strictly negative.
\end{proof}

\section{The mixed-edge lemma}\label{sec:edge}

\begin{lemma}[Mixed edge]\label{lem:edge}
Let $S\subset\N^2$ be finite and $\Delta=\mathrm{conv}(S)\subset\mathbb
R_{\ge0}^2$. Suppose the weight $w(a,b)=a-b$ takes both a positive and a
negative value on $S$. Then there is an integral linear form
$\ell(a,b)=ca+db$ with $c+d>0$ such that, writing
$\lambda=\min_{\alpha\in S}\ell(\alpha)$ and $F=\{\alpha\in S:\ell(\alpha)=\lambda\}$,
\begin{enumerate}
\item[(i)] $\lambda/(c+d)\ge0$, and
\item[(ii)] $w$ does not have one strict sign on $F$: either $w$ vanishes
somewhere on $F$, or $w$ takes both signs on $F$.
\end{enumerate}
\end{lemma}

\begin{proof}
\emph{Case 1: $\Delta$ is one-dimensional}, i.e.\ $S$ is contained in a line.
Let $(\Delta a,\Delta b)$ be a direction vector of that line, taken to be an
integral vector, and set $(c,d)=\pm(-\Delta b,\Delta a)$ with the sign chosen so
that $c+d=\pm(\Delta a-\Delta b)>0$; this is possible because $w$ is not
constant on $S$, so $\Delta a-\Delta b\ne0$. Then $\ell$ is constant on the
line, so $F=S$ and (ii) holds by hypothesis. For (i): $w$ takes both signs on
the segment $\Delta$ and is affine there, so $\Delta$ meets the diagonal at some
point $(t,t)$; since $\Delta\subset\mathbb R_{\ge0}^2$ we have $t\ge0$, and
$\lambda=\ell(t,t)=(c+d)t$, whence $\lambda/(c+d)=t\ge0$.

\emph{Case 2: $\Delta$ is two-dimensional.} The maximum of $w$ over $\Delta$
equals its maximum over $S$, which is positive, and is attained at a vertex;
similarly some vertex has $w<0$. Traversing the boundary of $\Delta$
counterclockwise, the continuous piecewise affine function $w$ therefore takes
both signs, so there is an edge $E$ from $p$ to $q$ (in the counterclockwise
order) with $w(p)\le0\le w(q)$ and $w(p)<w(q)$. Let
$(\Delta a,\Delta b)=q-p$, an integral vector, and put
$(c,d)=(-\Delta b,\Delta a)$, dividing by $\gcd$ to make it primitive. For a
counterclockwise-oriented convex polygon, $(-\Delta b,\Delta a)$ is the inward
normal of $E$, so $\ell$ attains its minimum over $\Delta$ exactly on $E$; in
particular $F=S\cap E$ and $\lambda=\ell(E)$. Its coordinate sum is
$c+d=\Delta a-\Delta b=w(q)-w(p)>0$. Since $w$ is affine on $E$ with
$w(p)\le0\le w(q)$, the edge meets the diagonal at some $(t,t)\in E$, and
$t\ge0$ because $E\subset\mathbb R_{\ge0}^2$; hence
$\lambda=\ell(t,t)=(c+d)t$ and $\lambda/(c+d)=t\ge0$, which is (i). Finally the
endpoints $p,q$ are vertices of $\Delta$, hence lie in $S$, hence lie in $F$;
as $w(p)\le0\le w(q)$, statement (ii) holds.
\end{proof}

\section{Proof of Theorem \ref{thm:main}}\label{sec:proof}

Let $P\in\C[X,Y]$ satisfy $\E(P^m)=0$ for every $m\ge1$, and let $S=\supp P$ in
the coordinates \eqref{eq:coords}. If $P=0$ there is nothing to prove, so assume
$S\ne\emptyset$.

\smallskip
\noindent\emph{Step 1: $S$ does not carry weights of both signs.} Suppose it
did. By Lemma \ref{lem:spec} there is $P'\in\Qbar[Z,W]$ with $\supp P'=S$ and
$\E(P'^m)=0$ for every $m\ge1$. Apply Lemma \ref{lem:edge} to $S$, obtaining an
integral form $\ell$ with $c+d>0$ and $\lambda/(c+d)\ge0$ whose exposed face $F$
carries a weight equal to zero, or weights of both signs. Apply Lemma
\ref{lem:isolation} to $P'$ and this $\ell$: all weights occurring on $F$ have
the same strict sign, the passage from the exponents of $R$ to the weights on
$F$ being justified by Remark \ref{rem:injective}. These two assertions about
$F$ contradict each other. Hence no polynomial over $\Qbar$ with support $S$
satisfies \eqref{eq:hyp}, and by Lemma \ref{lem:spec} none over $\C$ does
either (Remark \ref{rem:supportonly}). So every weight occurring in $P$ is
nonnegative, or every weight occurring in $P$ is nonpositive.

\smallskip
\noindent\emph{Step 2: the weight-zero part vanishes.} A monomial of weight
zero is $Z^aW^a=U^a$ with $U=ZW$, so the weight-zero part of $P$ is $C(U)$ for
some $C\in\C[U]$. By Step 1 all other monomials of $P$ have weight of one strict
sign, so a product of $m$ monomials of $P$ has weight zero only if every factor
is drawn from $C(U)$; hence the weight-zero part of $P^m$ is exactly $C(U)^m$.
By \eqref{eq:moments},
\[
  0=\E(P^m)=\mathcal L(C^m),\qquad \mathcal L(U^j)=j! \qquad (m\ge1).
\]
The one-variable case of the Factorial Conjecture, proved by van den Essen,
Wright and Zhao \cite[Thm.~4.9]{vdEWZ}, states that $\mathcal L(C^m)=0$ for
every $m\ge1$ forces $C=0$. Hence $P$ has no weight-zero monomial, and by Step
1 every monomial of $P$ has weight at least $1$, or every monomial has weight at
most $-1$.

\smallskip
\noindent\emph{Step 3: mixed moments.} Lemma \ref{lem:weights} now gives
$\E(QP^m)=0$ for every $Q\in\C[X,Y]$ and every $m>\deg Q$. \qed

\section{Remarks and corollaries}\label{sec:remarks}

\begin{remark}[Why the argument is specific to $n=2$]\label{rem:scope}
Since $\GMC(n)$ is false for $n\ge3$ \cite{Long}, it is worth identifying where
the proof uses $n=2$. With $r$ complex conjugate pairs $Z_j,W_j$ the weight
takes values in $\Z^r$, and the face weight polynomial $R$ is a Laurent
polynomial in $r$ variables. The Duistermaat--van der Kallen theorem still
applies and yields $0\notin\Newt(R)$, but for $r\ge2$ this is strictly weaker
than one-sidedness of the weights, so the contradiction in Step 1 is
unavailable; moreover Lemma \ref{lem:edge} is a statement about plane convex
geometry with no analogue in $\mathbb R^{r}$ for $r\ge2$. Step 2 also fails for
$r\ge2$, since the weight-zero part is then a polynomial in $U_1,\dots,U_r$ and
the multivariate Factorial Conjecture is not available. Both failures are
visible in Long's four-variable example
$P_4=(1+Z_2)\bigl(W_1(1-Z_1)+W_2\bigr)$, whose six monomials have weights
$(-1,0)$, $(0,-1)$, $(-1,1)$, $(0,0)$, $(0,0)$ and $(0,1)$: weight zero occurs,
and the weight-zero part is not a polynomial in a single $U$.
\end{remark}

\begin{corollary}[Quadratics]\label{cor:quadratic}
Let $P\in\C[X,Y]$ have total degree at most two and satisfy $\E(P^m)=0$ for
every $m\ge1$. Then $P=aZ+cZ^2$ or $P=bW+eW^2$ for some constants.
\end{corollary}

\begin{proof}
By Theorem \ref{thm:main} all monomials of $P$ have weight at least $1$, or all
at most $-1$. The monomials of degree at most two and weight at least $1$ are
exactly $Z$ and $Z^2$; those of weight at most $-1$ are exactly $W$ and $W^2$.
\end{proof}

\begin{corollary}[Contractions]\label{cor:contraction}
Let $H\in\C[Z]$ and $P=WH(Z)$ satisfy $\E(P^m)=0$ for every $m\ge1$. Then $H$ is
constant or $H$ vanishes to order at least two at the origin.
\end{corollary}

\begin{proof}
The monomials of $P$ are $h_jWZ^j$, of weight $j-1$. By Theorem \ref{thm:main}
either all occurring $j$ satisfy $j-1\ge1$, i.e.\ $\mathrm{ord}\,H\ge2$, or all
satisfy $j-1\le-1$, i.e.\ $j=0$ and $H$ is constant.
\end{proof}

Corollaries \ref{cor:quadratic} and \ref{cor:contraction} admit direct proofs
independent of the present argument ,  the first by a monodromy argument
applied to the closed form of $\E(e^{tP})$ for quadratic $P$, the second by the
observation that $\E(P^m)=m!\,[z^m]H(z)^m$ followed by an elementary induction
on the coefficients of $H$ ,  and the agreement of those proofs with Theorem
\ref{thm:main} is a useful consistency check on it.

\end{document}